\documentclass[11pt,draft]{article}
\usepackage{graphicx}
\usepackage{amssymb}
\usepackage{color,graphicx}
\usepackage{amsmath,amsthm,amscd}
\usepackage[latin1]{inputenc}
 \numberwithin{equation}{section}
\newtheorem{theorem}{Theorem}

\newtheorem{proposition}[theorem]{Proposition}

\newtheorem{definition}[theorem]{Definition}
\theoremstyle{definition}
\newtheorem{example}[theorem]{Example}

\newtheorem{remark}[theorem]{Remark}

\title{\textbf{Computing finite presentations \\ of $Tor$ and $Ext$ over skew $PBW$ extensions and some applications}}
\author{Oswaldo Lezama \& Melisa Paiba
\\
Seminario de Álgebra Constructiva - $\text{SAC}^2$\\
Departamento de Matemáticas\\
Universidad Nacional de Colombia, Bogotá, COLOMBIA\\
}
\date{}
\begin{document}
\maketitle
\begin{abstract}\noindent
In this paper we compute the $Tor$ and $Ext$ modules over skew $PBW$ extensions. If $A$ is a
bijective skew $PBW$ extension of a ring $R$, we give presentations of $Tor_r^{A}(M,N)$, where $M$
is a finitely generated centralizing subbimodule of $A^m$, $m\geq 1$, and $N$ is a left
$A$-submodule of $A^l$, $l\geq 1$. In the case of $Ext_A^{r}(M,N)$, $M$ is a left $A$-submodule of
$A^m$ and $N$ is a finitely generated centralizing subbimodule of $A^l$. As application of these
computations, we test stably-freeness, reflexiveness, and we will compute also the torsion, the
dual and the grade of a given submodule of $A^m$. Skew $PBW$ extensions include many important
classes of non-commutative rings and algebras arising in quantum mechanics, for example, Weyl
algebras, enveloping algebras of finite-dimensional Lie algebras (and its quantization), Artamonov
quantum polynomials, diffusion algebras, Manin algebra of quantum matrices, among many others.

\bigskip

\noindent \textit{Key words and phrases.} Skew $PBW$ extensions, noncommutative Gröbner bases,
module of syzygies, finite presentations of modules, $Ext$, $Tor$, stably-free modules, reflexive
modules, torsion, grade of a module.

\bigskip

\bigskip

\noindent 2010 \textit{Mathematics Subject Classification.} Primary: 16Z05. Secondary: 18G15.
\end{abstract}

\newpage

\section{Preliminaries}\label{sec12.5}

In this introductory section we recall the definition of skew $PBW$ extensions and summarize some
tools studied in previous works needed for the rest of the paper. In particular, we will review the
computation of the module of syzygies and finite presentations of modules over bijective skew $PBW$
extensions. The effective computation of these homological objects has been done in some recent
works using Gröbner bases. It is important to remark that the theory of Gröbner bases for ideals
and modules (left and right) has been completely constructed for this type of non-commutative
rings. The Gröbner theory for skew $PBW$ extensions and the details of the results presented in
this section can be consulted in \cite{lezama-gallego-projective}, \cite{gallego-thesis} and
\cite{Jimenez2} (see also \cite{Gallego2}).

The computations presented in the paper were inspired by the similar calculations done in the
beautiful papers \cite{Gomez-Torrecillas}, \cite{Gomez-Torrecillas2} and \cite{Gomez-Torrecillas6}
for $PBW$ rings and modules.

\subsection{Skew $PBW$ extensions}

We start recalling the definition of skew $PBW$ extensions; this class of non-commutative rings of
polynomial type was introduced in \cite{Gallego2}.

\begin{definition}\label{gpbwextension}
Let $R$ and $A$ be rings, we say that $A$ is a skew $PBW$ extension of $R$ {\rm(}also called
$\sigma-PBW$ extension{\rm)}, if the following conditions hold:
\begin{enumerate}
\item[\rm (i)]$R\subseteq A$.
\item[\rm (ii)]There exist finitely many elements $x_1,\dots ,x_n\in A$ such $A$ is a left $R$-free module with basis
\begin{center}
$Mon(A):=\{x_1^{\alpha_1}\cdots x_n^{\alpha_n}|\alpha=(\alpha_1,\dots ,\alpha_n)\in
\mathbb{N}^n\}$.
\end{center}
\item[\rm (iii)]For every $1\leq i\leq n$ and $r\in R-\{0\}$ there exists $c_{i,r}\in R-\{0\}$ such that
\begin{equation*}\label{sigmadefinicion1}
x_ir-c_{i,r}x_i\in R.
\end{equation*}
\item[\rm (iv)]For every $1\leq i,j\leq n$ there exists $c_{i,j}\in R-\{0\}$ such that
\begin{equation*}\label{sigmadefinicion2}
x_jx_i-c_{i,j}x_ix_j\in R+Rx_1+\cdots +Rx_n.
\end{equation*}
Under these conditions we will write $A=\sigma(R)\langle x_1,\dots ,x_n\rangle$.
\end{enumerate}
\end{definition}

\begin{proposition}\label{sigmadefinition}
Let $A$ be a skew $PBW$ extension of $R$. Then, for every $1\leq i\leq n$, there exist an injective
ring endomorphism $\sigma_i:R\rightarrow R$ and a $\sigma_i$-derivation $\delta_i:R\rightarrow R$
such that
\begin{center}
$x_ir=\sigma_i(r)x_i+\delta_i(r)$,
\end{center}
for each $r\in R$.
\end{proposition}
\begin{proof}
See \cite{Gallego2}.
\end{proof}
Two important particular cases of skew $PBW$ extensions are the following.
\begin{definition}\label{sigmapbwderivationtype}
Let $A$ be a skew $PBW$ extension.
\begin{enumerate}
\item[\rm (a)]
$A$ is quasi-commutative if the conditions {\rm(}iii{\rm)} and {\rm(}iv{\rm)} in Definition
\ref{gpbwextension} are replaced by
\begin{enumerate}
\item[\rm ($iii'$)]For every $1\leq i\leq n$ and $r\in R-\{0\}$ there exists $c_{i,r}\in R-\{0\}$ such that
\begin{equation*}
x_ir=c_{i,r}x_i.
\end{equation*}
\item[\rm ($iv'$)]For every $1\leq i,j\leq n$ there exists $c_{i,j}\in R-\{0\}$ such that
\begin{equation*}
x_jx_i=c_{i,j}x_ix_j.
\end{equation*}
\end{enumerate}
\item[\rm (b)]$A$ is bijective if $\sigma_i$ is bijective for
every $1\leq i\leq n$ and $c_{i,j}$ is invertible for any $1\leq i<j\leq n$.
\end{enumerate}
\end{definition}
\begin{remark}\label{rem11.5.7}
(i) Skew $PBW$ extensions include many important classes of non-commutative rings and algebras
arising in quantum mechanics, for example, Weyl algebras, enveloping algebras of finite-dimensional
Lie algebras, Artamonov quantum polynomials, diffusion algebras, Manin algebra of quantum matrices.
An extensive list of examples of skew $PBW$ extensions can be found in \cite{Lezama3}. Some other
authors have classified and characterized quantum algebras and other non-commutative rings of
polynomial type by similar notions: In \cite{Levandovskyy} are defined the $G$-algebras, Bueso,
Gómez-Torrecillas and Verschoren in \cite{Gomez-Torrecillas2} introduced the $PBW$ rings, Panov in
\cite{Panov} defined the so called $Q$-solvable algebras.

(ii) In Definition \ref{gpbwextension} we assumed that $A$ is a left $R$-module; note that we can
also define right skew $PBW$ extensions adapting Definition \ref{gpbwextension} with the elements
of $R$ on the right and assuming that $A$ is free right $R$-module with basis $Mon(A)$. However,
note that $A$ is a bijective left skew $PBW$ extension of $R$ if and only if $A$ is bijective right
skew $PBW$ extension of $R$ (see \cite{gallego-thesis}). Observe that in the right case, the
automorphisms in Proposition \ref{sigmadefinition} are $\sigma_i^{-1}$ and the constants in numeral
(iv) of Definition \ref{gpbwextension} are $c_{i,j}^{-1}$, $1\leq i,j\leq n$.

(iii) With respect to the Gröbner theory it is necessary to make the following remark: if $A$ is a
bijective skew $PBW$ extension of a ring $R$ (we mean left if nothing contrary is said), then $A$
is also a bijective right skew $PBW$ extension of $R$, and therefore, we have a left and a right
division algorithm. Obviously, if the elements of $A$ are given by their left standard polynomial
representation, we may have to rewrite them in their right standard polynomial representation in
order to be able to perform right divisions. Left and right versions of Buchberger's algorithm are
also available. Thus, the theory of Gröbner bases has its right counterpart for $A$ (see
\cite{gallego-thesis}).
\end{remark}

\subsection{Syzygy of a module}\label{31}

In this paper, if nothing contrary is assumed, $A=\sigma(R)\langle x_1,\dots,x_n\rangle$ is a
bijective skew $PBW$ extension of $R$;  $A^m$, $m\geq 1$, represents the left free $A$-module of
column vectors of length $m\geq 1$. Since $A$ is bijective, $A$ is a left Noetherian ring (see
\cite{Lezama3}), and hence $A$ is an $IBN$ ring (\textit{Invariant Basis Number}), so all bases of
the free module $A^m$ have $m$ elements. Note moreover that $A^m$ is Noetherian, and hence, any
submodule of $A^m$ is finitely generated.

In this section we recall the definition of the module of syzygies of a submodule $M=\langle
\textbf{\emph{f}}_1,\dots ,\textbf{\emph{f}}_s\rangle$ of $A^m$. Let $f$ be the canonical
homomorphism defined by
\begin{align*}
A^s& \xrightarrow{f} A^m\\
\textbf{\emph{e}}_j& \mapsto \textbf{\emph{f}}_j
\end{align*}
where $\{\textbf{\emph{e}}_1,\dots,\textbf{\emph{e}}_s\}$ is the
canonical basis of $A^s$. Observe that $f$ can be represented by a
matrix, i.e., if $\textbf{\emph{f}}_j:=(f_{1j},\dots,f_{mj})^T$,
then the matrix of $f$ in the canonical bases of $A^s$ and $A^m$
is
\begin{center}
$F:=\begin{bmatrix}\textbf{\emph{f}}_1 & \cdots &
\textbf{\emph{f}}_s\end{bmatrix}=
\begin{bmatrix}
f_{11} & \cdots & f_{1s}\\
\vdots & & \vdots\\
f_{m1} & \cdots & f_{ms}
\end{bmatrix}\in M_{m\times s}(A)$.
\end{center}
Note that $Im(f)$ is the column module of $F$, i.e., the left
$A$-module generated by the columns of $F$:
\begin{center}
$Im(f)=\langle
f(\textbf{\emph{e}}_1),\dots,f(\textbf{\emph{e}}_s)\rangle=\langle
\textbf{\emph{f}}_1,\dots ,\textbf{\emph{f}}_s \rangle =\langle F
\rangle$.
\end{center}
Moreover, observe that if $\textbf{\emph{a}}:=(a_1,\dots,a_s)^T\in
A^s$, then
\begin{equation*}\label{equ1.5}
f(\textbf{\emph{a}})=(\textbf{\emph{a}}^TF^T)^T.
\end{equation*}
By definition
\begin{center}
$Syz(\{\textbf{\emph{f}}_1,\dots
,\textbf{\emph{f}}_s\}):=\{\textbf{\emph{a}}:=(a_1,\dots,a_s)^T\in
A^s|a_1\textbf{\emph{f}}_1+\cdots+a_s\textbf{\emph{f}}_s=\textbf{0}\}$.
\end{center}
Note that
\begin{equation*}
Syz(\{\textbf{\emph{f}}_1,\dots ,\textbf{\emph{f}}_s\})=\ker(f),
\end{equation*}
but $Syz(\{\textbf{\emph{f}}_1,\dots ,\textbf{\emph{f}}_s\})\neq
\ker(F)$ since we have
\begin{equation*}\label{313}
\textbf{\emph{a}}\in Syz(\{\textbf{\emph{f}}_1,\dots
,\textbf{\emph{f}}_s\})\Leftrightarrow
\textbf{\emph{a}}^TF^T=\textbf{0}.
\end{equation*}
The modules of syzygies of $M$ and $F$ are defined by
\begin{equation*}
Syz(M):=Syz(F):=Syz(\{\textbf{\emph{f}}_1,\dots
,\textbf{\emph{f}}_s\}).
\end{equation*}
The generators of $Syz(F)$ can be disposed into a matrix, so
sometimes we will refer to $Syz(F)$ as a matrix. Thus, if $Syz(F)$
is generated by $r$ vectors, $\textbf{\emph{z}}_1, \dots,
\textbf{\emph{z}}_r$, then
\begin{center}
$Syz(F)=\langle \textbf{\emph{z}}_1, \dots,
\textbf{\emph{z}}_r\rangle$,
\end{center}
and we will use also the following matrix notation
\begin{center}
$Syz(F)=\begin{bmatrix}\textbf{\emph{z}}_1 & \cdots & \textbf{\emph{z}}_r
\end{bmatrix}=\begin{bmatrix}z_{11} & \cdots & z_{1r}\\
\vdots & & \vdots\\
z_{s1} & \cdots & z_{sr}\end{bmatrix}\in M_{s\times r}(A)$.
\end{center}

\subsection{Presentation of a module}

\noindent Let $M=\langle \textbf{\emph{f}}_1,\dots ,\textbf{\emph{f}}_s\rangle$ be a submodule of
$A^m$, there exists a natural surjective homomorphism $\pi_M:A^s\longrightarrow M$ defined by
$\pi_M (\textbf{\emph{e}}_i):=\textbf{\emph{f}}_i$, where $\{\textbf{\emph{e}}_i\}_{1\leq i\leq s}$
is the canonical basis of $A^s$. We have the isomorphism $\overline{\pi_M}:A^s/\ker(\pi_M)\cong M$,
defined by $\overline{\pi_M}(\overline{\textbf{\emph{e}}_i}):=\textbf{\emph{f}}_i$, where
$\overline{\textbf{\emph{e}}_i}:=\textbf{\emph{e}}_i+\ker(\pi_M)$. Since $A^s$ is a Noetherian $A$-
module, $\ker(\pi_M)$ is also a finitely generated module, $\ker(\pi_M):=\langle
\textbf{\emph{h}}_1,\dots ,\textbf{\emph{h}}_{s_1}\rangle$, and hence, we have the exact sequence
\begin{equation}\label{presentation}
A^{s_1}\xrightarrow {\delta_M}A^s\xrightarrow {\pi_M} M\longrightarrow 0,
\end{equation}
with $\delta_M:=l_M\circ \pi_M'$, where $l_M$ is the inclusion of $\ker(\pi_M)$ in $A^s$ and
$\pi_M'$ is the natural surjective homomorphism from $A^{s_1}$ to $\ker(\pi_M)$. We note that
$\ker(\pi_M)=Syz(M)=Syz(F)$, where $F=[\textbf{\emph{f}}_1 \cdots \textbf{\emph{f}}_s]\in M_{m
\times s}(A)$
\begin{definition}
It is said that $A^s/Syz(M)$ is a presentation of $M$. It is said also that the sequence
{\rm(}\ref{presentation}{\rm)} is a finite presentation of $M$, and $M$ is a finitely presented
module.
\end{definition}
Let $\Delta_M$ be the matrix of $\delta_M$ in the canonical bases of $A^{s_1}$ and $A^s$; since $Im
(\delta_M)=\ker (\pi_M)$, then
\[\Delta_M =
\begin{bmatrix}
\boldsymbol{h}_{1} &\cdots &\boldsymbol{h}_{s_1}
\end{bmatrix}
=
\begin{bmatrix}
h_{11} &\cdots &h_{1s_1}\\
\vdots & &\vdots\\
h_{s1} &\cdots &h_{ss_1}
\end{bmatrix}
\in M_{s \times s_1}(A),\] and hence, the columns of $\Delta_M$ are the generators of $Syz(F)$,
i.e.,
\begin{center}
$\Delta_M = Syz(F)$.
\end{center}
\begin{definition}
With the previous notation, it is said that $\Delta_M$ is a matrix presentation of $M$.
\end{definition}
We can also compute presentations of quotient modules. Indeed, let $N\subseteq M$ be submodules of
$A^m$, where $M=\langle \textbf{\emph{f}}_1,\dots ,\textbf{\emph{f}}_s\rangle$, $N=\langle
\textbf{\emph{g}}_1,\dots, \textbf{\emph{g}}_t\rangle$ and $M/N=$ $\langle
\overline{\textbf{\emph{f}}_1},\dots ,\overline{\textbf{\emph{f}}_s}\rangle$, then we have a
canonical surjective homomorphism $A^s\longrightarrow M/N$ such that a presentation of $M/N$ is
given by
\begin{center}
$M/N\cong A^s/Syz(M/N)$.
\end{center}
But $Syz(M/N)$ can be computed in the following way:
$\textbf{\emph{h}}=(h_1,\dots ,h_s)^T\in Syz(M/N)$ if and only if $h_1\textbf{\emph{f}}_1+\cdots
+h_s\textbf{\emph{f}}_s\in \langle \textbf{\emph{g}}_1,\dots ,\textbf{\emph{g}}_t\rangle$ if and
only if there exist $h_{s+1},\dots ,h_{s+t}\in A$ such that $h_1\textbf{\emph{f}}_1+\cdots
+h_s\textbf{\emph{f}}_s+h_{s+1}\textbf{\emph{g}}_1+\cdots +h_{s+t}\textbf{\emph{g}}_t=\textbf{0}$
if and only if $(h_1,\dots,h_s,$ $h_{s+1},\dots ,h_{s+t})\in Syz(H)$, where
\begin{center}
$H:=[\textbf{\emph{f}}_1 \,\, \cdots \,\, \textbf{\emph{f}}_s \, \, \textbf{\emph{g}}_1\,\, \cdots
\,\,\textbf{\emph{g}}_t]$.
\end{center}
\begin{proposition}\label{quotient}
With the notation above, a presentation of $M/N$ is given by $A^s/Syz(M/N)$, where a set of
generators of $Syz(M/N)$ are the first $s$ coordinates of generators of $Syz(H)$. Thus, a finite
presentation of $M/N$ is effective computable.
\end{proposition}
\begin{remark}\label{12.7.7}
(i) A procedure for computing $Syz(M)$ using Gröbner bases can be found in \cite{gallego-thesis}
and \cite{Jimenez2}. Thus, since the module $Syz(M)$ is computable, then a presentation of $M$ is
computable as well as $\Delta_M$ and $Syz(M/N)$

(ii) From Remark \ref{rem11.5.7} we conclude that the applications established for left  modules
have also their right version. Thus, if $A$ is a bijective skew $PBW$ extension, then $A$ is a
right $PBW$ extension; the right $A$-module of syzygies of a right $A$-submodule $M$ of the free
right $A$-module $A^m$ will be denoted by $Syz^r(M)$, and all results about right syzygies and
presentations are valid. For example, if $M=\langle \textbf{\emph{f}}_1,\dots
,\textbf{\emph{f}}_s\rangle$ is a submodule of the right $A$-module $A^m$ and $\pi_M:A^s\to M$ is
the canonical homomorphism defined by $\pi_M(\textbf{\emph{e}}_i):=\textbf{\emph{f}}_i$, $1\leq
i\leq s$, then $f(\textbf{\emph{a}})=F\textbf{\emph{a}}$ and
\begin{center}
$\ker(\pi_M)=Syz^r(M)=Syz^r(\{\textbf{\emph{f}}_1,\dots ,\textbf{\emph{f}}_s\})=Syz^r(F)=\ker(F)$,
\end{center}
with $F:=\begin{bmatrix}\textbf{\emph{f}}_1 & \cdots & \textbf{\emph{f}}_s\end{bmatrix}$.

\end{remark}

\subsection{Centralizing bimodules}

\noindent We conclude this introductory section with another tool needed for proving the main
results of the paper.

\begin{definition}[\cite{Gomez-Torrecillas2}]
Let $M$ be an $A$-bimodule. The centralizer of $M$ is defined by
\begin{center}
$Cen_A(M):=\{\textbf{f}\in M|\textbf{f}\,a=a\textbf{f}, \ \text{for any $a\in A$}\}$.
\end{center}
It is said that $M$ is centralizing if $M$ is generated as left $A$-module $($or, equivalently, as
right $A$-module$)$ by its centralizer. $M$ is a finitely generated centralizing $A$-bimodule if
there exists a finite set of elements in $Cen_A(M)$ that generates $M$.
\end{definition}

\begin{example}
Let $M:=_A\langle \textbf{\emph{f}}_1,\dots, \textbf{\emph{f}}_s\rangle_A$ be a finitely generated
centralizing $A$-subbimodule of $A^m$; note that every $\textbf{\emph{f}}_i$ has the form
$\textbf{\emph{f}}_i=(a_1,\dots,a_m)$, with $a_j\in A$, $1\leq j\leq m$; therefore $aa_j=a_ja$ for
any $a\in A$, i.e., $a_j\in Z(A)$. The center of most of non-trivial skew $PBW$ extensions is too
small, so it is not easy to give enough non-trivial examples of this type of bimodules. However,
for the applications we will show, the centralizing bimodule is $A=_A\langle 1\rangle_A$. Anyway,
next we present an example of finitely generated centralizing subbimodule:  Consider the enveloping
algebra $A:=\mathcal{U}(\mathcal{G})$ of the $3$-dimensional Lie algebra $\mathcal{G}$ over a filed
$K$ with
\begin{center}
$[x_1,x_2]=0=[x_1,x_3]$ and $[x_2,x_3]=x_1$.
\end{center}
Observe that $x_1\in Z(\mathcal{U}(\mathcal{G}))$, and hence,
\begin{center}
$_A\langle(x_1, x_1^2+1,1),(1,x_1-1,x_1^2)\rangle_A$
\end{center}
is a centralizing subbimodule of $A^3$.

\end{example}

Since we have to consider subbimodules we will use the following special notation from now on:
$_A\langle X\rangle$ denotes the left $A$-module generated by $X$, $\langle X\rangle_A$ the right
$A$-module generated by $X$ and $_A\langle X\rangle_A$ the bimodule generated by $X$; if $X$ is
contained in the centralizer, then $_A\langle X\rangle_A=_A\langle X\rangle=\langle X\rangle_A$.

\section{Computing $Tor$ and $Ext$}

Now we pass to consider the main computations of the present paper.
\subsection{Computation of $M\otimes N$}

\noindent We start computing a presentation of the left $A$-module $M\otimes_A N$, where
$M:=_A\langle \textbf{\emph{f}}_1,\dots, \textbf{\emph{f}}_s\rangle_A$ is a finitely generated
centralizing $A$-subbimodule of $A^m$, i.e., $\textbf{\emph{f}}_i\in Cen_A(M)$, $1\leq i\leq s$,
and $N:=_A\langle \textbf{\emph{g}}_1,\dots,\textbf{\emph{g}}_t\rangle$ is a finitely generated
left $A$-submodule of $A^l$. Our presentation is given as a quotient module of $M\otimes_A N$ and
we will show an explicit set of generators for $Syz(M\otimes_A N)$. We will simplify the notation
writing $M\otimes N$ instead of $M\otimes_A N$. A preliminary well known result is needed.

\begin{proposition}\label{tensorpreliminary}Let $S$ be an arbitrary ring, $M$ be a right module over $S$ and $N$ be a left $S$-module.
Let $\textbf{m}_j\in M,\textbf{g}_j\in N$, $1\leq j\leq t$, such that $N=_A\langle
\textbf{g}_1,\dots ,\textbf{g}_t\rangle$. Then, $\textbf{m}_1\otimes \textbf{g}_1+\cdots
+\textbf{m}_t\otimes \textbf{g}_t=\textbf{\emph{0}}$ if and only if there exist elements
$\textbf{m}_v'\in M$, $1\leq v\leq r$, and a matrix $H:=[h_{jv}]\in M_{t\times r}(S)$, such that
\begin{center}
$[\textbf{m}_1 \cdots \textbf{m}_t]=[\textbf{m}_1' \cdots \textbf{m}_r']H^{T},\,\, H^T[\textbf{g}_1
\cdots \textbf{g}_t]^T=0$.
\end{center}
Thus, the left $S$-module expanded by the columns of $H$ is contained in $Syz(N)$.
\end{proposition}

\begin{theorem}\label{tensor}
Let $M:=_A\langle \textbf{f}_1,\dots, \textbf{f}_s\rangle_A$ be a finitely generated centralizing
$A$-subbimodule of $A^m$, with $\textbf{f}_i\in Cen_A(M)$, $1\leq i\leq s$, and $N:=_A\langle
\textbf{g}_1,\dots,\textbf{g}_t\rangle$ be a finitely generated left $A$-submodule of $A^l$. Then,
\begin{equation*}
M\otimes N\cong A^{st}/Syz(M\otimes N),
\end{equation*}
where
\begin{equation*}
Syz(M\otimes N)=_A\langle [Syz(M)\otimes I_t\mid I_s\otimes Syz(N)]\rangle.
\end{equation*}
\end{theorem}
\begin{proof}
Since every $\textbf{\emph{f}}_i\in Cen_A(M)$, it is clear that
\begin{equation*}
M\otimes N=_A\langle \textbf{\emph{f}}_1\otimes \textbf{\emph{g}}_1,\dots
,\textbf{\emph{f}}_1\otimes \textbf{\emph{g}}_t,\dots ,\textbf{\emph{f}}_s\otimes
\textbf{\emph{g}}_1,\dots ,\textbf{\emph{f}}_s\otimes \textbf{\emph{g}}_t\rangle.
\end{equation*}
Let $Syz(M):=\langle \textbf{\emph{f}}_1',\dots , \textbf{\emph{f}}_r'\rangle$ be the left
$A$-module of syzygies of the left $A$-module $M$, and $Syz(N):=\langle \textbf{\emph{g}}_1',\dots
, \textbf{\emph{g}}_p'\rangle$ be the left $A$-module of syzygies of the left $A$-module $N$, with
\begin{equation*}
\textbf{\emph{f}}_1':=(f_{11},\dots ,f_{s1})^T,\dots , \textbf{\emph{f}}_r':=(f_{1r},\dots
,f_{sr})^T
\end{equation*}
and
\begin{equation*}
\textbf{\emph{g}}_1':=(g_{11},\dots ,g_{t1})^T,\dots , \textbf{\emph{g}}_p':=(g_{1p},\dots
,g_{tp})^T.
\end{equation*}
In a matrix notation,
\begin{center}
$Syz(M)=
\begin{bmatrix}
f_{11} & \dots & f_{1r} \\
\vdots & \dots & \vdots \\
f_{s1} & \dots & f_{sr}
\end{bmatrix},\,\, Syz(N)=
\begin{bmatrix}
g_{11} & \dots & g_{1p} \\
\vdots & \dots & \vdots \\
g_{t1} & \dots & g_{tp}
\end{bmatrix}$.
\end{center}

Then,
\begin{center}
$
\begin{matrix}
f_{11}\textbf{\emph{f}}_1+\cdots +f_{s1}\textbf{\emph{f}}_s & = &\textbf{0}\\
\vdots & & \\
f_{1r}\textbf{\emph{f}}_1+\cdots +f_{sr}\textbf{\emph{f}}_s & = &\textbf{0}
\end{matrix}
$
\end{center}
and
\begin{center}
$
\begin{matrix}
g_{11}\textbf{\emph{g}}_1+\cdots +g_{t1}\textbf{\emph{g}}_t & = &\textbf{0}\\
\vdots & & \\
g_{1p}\textbf{\emph{g}}_1+\cdots +g_{tp}\textbf{\emph{g}}_t & = &\textbf{0}.
\end{matrix}
$
\end{center}
We note that any of the following $tr$ vectors has $st$ entries and belongs to $Syz(M\otimes N)$
\begin{center}
$
\begin{matrix}
(f_{11},0,\dots ,0,\dots ,f_{s1},0,\dots ,0)^T \\
\vdots \\
(0,0,\dots ,f_{11},\dots ,0,0,\dots ,f_{s1})^T\\
\vdots \\
(f_{1r},0,\dots ,0,\dots ,f_{sr},0,\dots ,0)^T \\
\vdots \\
(0,0,\dots ,f_{1r},\dots ,0,0,\dots ,f_{sr})^T.
\end{matrix}
$
\end{center}
Since for every $1\leq i\leq s$, $\textbf{\emph{f}}_i\in Cen_A(M)$, any of the following $ps$
vectors has $st$ entries and belongs to $Syz(M\otimes N)$
\begin{center}
$
\begin{matrix}
(g_{11},\dots ,g_{t1},\dots ,0,\dots ,0)^T \\
\vdots \\
(0,\dots ,0,\dots ,g_{11},\dots ,g_{t1})^T \\
\vdots \\
(g_{1p},\dots ,g_{tp},\dots ,0,\dots ,0)^T \\
\vdots \\
(0,\dots ,0,\dots ,g_{1p},\dots ,g_{tp})^T.
\end{matrix}
$
\end{center}
We can dispose these $tr+ps$ vectors by columns in a matrix $[C\mid B]$ of size $st\times (tr+ps)$,
where
\begin{center}
$C:=
\begin{bmatrix}
f_{11} & \dots & 0 & \dots & f_{1r} & \dots & 0 \\
 & \ddots &  & \dots &  & \ddots &  \\
0 & \dots & f_{11} & \dots & 0 & \dots & f_{1r} \\
 & \vdots &  & \dots &  & \vdots &  \\
f_{s1} & \dots & 0 & \dots & f_{sr} & \dots & 0 \\
 & \ddots &  & \dots &  & \ddots &  \\
0 & \dots & f_{s1} & \dots & 0 & \dots & f_{sr}
\end{bmatrix}=Syz(M)\otimes I_t
$
\end{center}
\begin{center}
$B:=
\begin{bmatrix}
g_{11} & \dots & 0 & \dots & g_{1p} & \dots & 0\\
\vdots & \vdots & \vdots & \dots & \vdots & \vdots & \vdots \\
g_{t1} & \dots & 0 & \dots & g_{tp} & \dots & 0\\
\vdots  & \ddots & \vdots & \dots & \vdots & \ddots & \vdots\\
0 & \dots & g_{11} & \dots & 0 & \dots & g_{1p}\\
\vdots & \vdots & \vdots & \dots & \vdots & \vdots & \vdots \\
0 & \dots & g_{t1} & \dots & 0 & \dots & g_{tp}
\end{bmatrix}.
$
\end{center}
But $B$ can be changed by
\begin{center}
$
\begin{bmatrix}
g_{11} & \dots & g_{1p} & \dots & 0 & \dots & 0\\
\vdots & \vdots & \vdots & \dots & \vdots & \vdots & \vdots \\
g_{t1} & \dots & g_{tp} & \dots & 0 & \dots & 0\\
\vdots  & \vdots & \vdots & \ddots & \vdots & \vdots & \vdots\\
0 & \dots & 0 & \dots & g_{11} & \dots & g_{1p}\\
\vdots & \vdots & \vdots & \dots & \vdots & \vdots & \vdots \\
0 & \dots & 0 & \dots & g_{t1} & \dots & g_{tp}
\end{bmatrix}=I_s\otimes Syz(N).
$
\end{center}
Thus, we have proved that $_A\langle [Syz(M)\otimes I_t\mid I_s\otimes Syz(N)]\rangle\subseteq
Syz(M\otimes N)$.

Now we assume that $\textbf{\emph{h}}:=(h_{11},\dots ,h_{1t}, \dots , h_{s1},\dots ,h_{st})^T\in
Syz(M\otimes N)$, then
\begin{center}
$h_{11}(\textbf{\emph{f}}_1\otimes \textbf{\emph{g}}_1)+\cdots +h_{1t}(\textbf{\emph{f}}_1\otimes
\textbf{\emph{g}}_t) +\cdots + h_{s1}(\textbf{\emph{f}}_s\otimes \textbf{\emph{g}}_1)+\cdots
+h_{st}(\textbf{\emph{f}}_s\otimes \textbf{\emph{g}}_t)=\textbf{0}$.
\end{center}
From this we get that
\begin{center}
$\textbf{\emph{m}}_1\otimes \textbf{\emph{g}}_1+\cdots +\textbf{\emph{m}}_t\otimes
\textbf{\emph{g}}_t=\textbf{0}$,
\end{center}
where
\begin{center}
$\textbf{\emph{m}}_j=h_{1j}\textbf{\emph{f}}_1+\cdots +h_{sj}\textbf{\emph{f}}_s\in M$,
\end{center}
with $1\leq j\leq t$. From Proposition \ref{tensorpreliminary}, there exist polynomials $a_{jv}\in
A$ and vectors $\textbf{\emph{m}}_v'\in M$ such that
$\textbf{\emph{m}}_j=\sum_{v=1}^{r}a_{jv}\textbf{\emph{m}}_v'$ and
$\sum_{j=1}^{t}a_{jv}\textbf{\emph{g}}_j=\textbf{0}$, for each $1\leq v\leq r$. This means that
$(a_{1v},\dots ,a_{tv})\in Syz(N)$ for each $1\leq v\leq r$. Since $\textbf{\emph{m}}_v'\in M$
there exist $q_{uv}\in A$ such that $\textbf{\emph{m}}_v'=q_{1v}\textbf{\emph{f}}_1+\cdots
+q_{sv}\textbf{\emph{f}}_s$, and then
\begin{center}
$
\begin{matrix}
\sum_{i=1}^sh_{i1}\textbf{\emph{f}}_i & = & a_{11}(q_{11}\textbf{\emph{f}}_1
+\cdots +q_{s1}\textbf{\emph{f}}_s)+\cdots +a_{1r}(q_{1r}\textbf{\emph{f}}_1+\cdots +q_{sr}\textbf{\emph{f}}_s)\\
 & \vdots & \\
\sum_{i=1}^sh_{it}\textbf{\emph{f}}_i & = & a_{t1}(q_{11}\textbf{\emph{f}}_1+\cdots
+q_{s1}\textbf{\emph{f}}_s)+\cdots +a_{tr}(q_{1r}\textbf{\emph{f}}_1+\cdots
+q_{sr}\textbf{\emph{f}}_s).
\end{matrix}
$
\end{center}
From this we get that
\begin{center}
$
\begin{matrix}
\sum_{i=1}^s(h_{i1}-(a_{11}q_{i1}+\cdots +a_{1r}q_{ir}))\textbf{\emph{f}}_i & = & \textbf{0}\\
\vdots  & & \\
\sum_{i=1}^s(h_{it}-(a_{t1}q_{i1}+\cdots +a_{tr}q_{ir}))\textbf{\emph{f}}_i & = & \textbf{0}
\end{matrix}
$
\end{center}
i.e.,
\begin{center}
$
\begin{matrix}
(h_{11}-(a_{11}q_{11}+\cdots +a_{1r}q_{1r})),\dots ,h_{s1}-(a_{11}q_{s1}+\cdots +a_{1r}q_{sr})) \in Syz(M)\\
\vdots  \\
(h_{1t}-(a_{t1}q_{11}+\cdots +a_{tr}q_{1r})),\dots ,h_{st}-(a_{t1}q_{s1}+\cdots +a_{tr}q_{sr}))\in
Syz(M).
\end{matrix}
$
\end{center}
This implies that
\begin{center}
$
\begin{matrix}
(h_{11},\dots ,h_{s1})-(a_{11}q_{11}+\cdots +a_{1r}q_{1r},\dots ,a_{11}q_{s1}+\cdots +a_{1r}q_{sr})\in Syz(M)\\
\vdots \\
(h_{1t},\dots ,h_{st})-(a_{t1}q_{11}+\cdots +a_{tr}q_{1r},\dots ,a_{t1}q_{s1}+\cdots
+a_{tr}q_{sr})\in Syz(M).
\end{matrix}
$
\end{center}
Then,
\begin{center}
$
\begin{matrix}
(h_{i1})_{i=1}^s=(a_{11}q_{11}+\cdots +a_{1r}q_{1r},\dots ,a_{11}q_{s1}+\cdots +a_{1r}q_{sr})+(f_{11},\dots ,f_{s1})\\
\vdots \\
(h_{it})_{i=1}^s =(a_{t1}q_{11}+\cdots +a_{tr}q_{1r},\dots ,a_{t1}q_{s1}+\cdots
+a_{tr}q_{sr})+(f_{1t},\dots ,f_{st}),
\end{matrix}
$
\end{center}
with $(f_{11},\dots ,f_{s1}),\dots ,(f_{1t},\dots ,f_{st})\in Syz(M)$. From this we get
\begin{center}
$
\begin{matrix}
h_{11}=f_{11}+a_{11}q_{11}+\cdots +a_{1r}q_{1r}\\
\vdots \\
h_{1t}=f_{1t}+a_{t1}q_{11}+\cdots +a_{tr}q_{1r}\\
\vdots \\
h_{s1}=f_{s1}+a_{11}q_{s1}+\cdots +a_{1r}q_{sr}\\
\vdots \\
h_{st}=f_{st}+a_{t1}q_{s1}+\cdots +a_{tr}q_{sr},
\end{matrix}
$
\end{center}
and hence $\textbf{\emph{h}}$ is a linear combination of the columns of the following matrix
\begin{center}
$[Syz(M)\otimes I_t\mid D]$,
\end{center}
where
\begin{center}
$D=
\begin{bmatrix}
a_{11} & \dots & a_{1r} & \dots & 0 & \dots & 0\\
\vdots & \vdots & \vdots & \dots & \vdots & \vdots & \vdots \\
a_{t1} & \dots & a_{tr} & \dots & 0 & \dots & 0\\
\vdots  & \vdots & \vdots & \ddots & \vdots & \vdots & \vdots\\
0 & \dots & 0 & \dots & a_{11} & \dots & a_{1r}\\
\vdots & \vdots & \vdots & \dots & \vdots & \vdots & \vdots \\
0 & \dots & 0 & \dots & a_{t1} & \dots & a_{tr}
\end{bmatrix}.
$
\end{center}
But $_A\langle D\rangle\subseteq \langle I_s\otimes Syz(N)\rangle$, and hence, $h\in _A\langle
[Syz(M)\otimes I_t\mid I_s\otimes Syz(N)]\rangle$. This completes the proof of the theorem.
\end{proof}

\subsection{Computation of $Tor$}

Using syzygies we present next an easy procedure for computing the left $A$-modules
$Tor^A_{r}(M,N)$ for $r\geq 0$, where $M:=_A\langle \textbf{\emph{f}}_1,\dots,
\textbf{\emph{f}}_s\rangle_A$ is a finitely generated centralizing $A$-subbimodule of $A^m$, i.e.,
$\textbf{\emph{f}}_i\in Cen_A(M)$, $1\leq i\leq s$, and $N:=_A\langle
\textbf{\emph{g}}_1,\dots,\textbf{\emph{g}}_t\rangle$ is a finitely generated left $A$-submodule of
$A^l$. By computing we mean to find a presentation and a system of generators of $Tor^A_{r}(M,N)$.
Our computations of course extended the well known results on commutative polynomial rings, see for
example \cite{Pfister}, Proposition 7.1.3. For $r=0$, the computation was given in Theorem
\ref{tensor}. So we assume that $r\geq 1$.

\textit{Presentation of} $Tor^A_{r}(M,N)$, $r\geq 1$:

\textit{Step 1}. We first compute left presentations of $M$ and $N$,
\begin{center}
$M\cong A^s/Syz(M),N\cong A^t/Syz(N)$;
\end{center}
remember that $Syz(M)$ and $Syz(N)$ are the kernels of the natural left $A$-\-ho\-mo\-mor\-phisms
\begin{center}
$\pi_M:A^s\longrightarrow M$ and $\pi_N:A^t\longrightarrow N$
\end{center}
defined by $\pi_M(\textbf{\emph{e}}_i):=\textbf{\emph{f}}_i$,
$\pi_N(\textbf{\emph{e}}^{\prime}_j):=\textbf{\emph{g}}_j$, $1\leq i\leq s$, $1\leq j \leq t$,
where $\{\textbf{\emph{e}}_i\}_{1\leq i\leq s}$ is the canonical basis of $A^m$ and
$\{\textbf{\emph{e}}_j'\}_{1\leq j \leq t}$ is the canonical basis of $A^t$ (see Section
\ref{sec12.5}). However, since every generator $\textbf{\textit{f}}_i$ of $M$ is in $Cen_A(M)$,
then $\pi_M$ is a $A$-bimodule homomorphism and hence $Syz(M)$ is a subbimodule of $A^s$. Thus,
$A^s/Syz(M)$ is an $A$-bimodule.

\textit{Step 2}. We compute a free resolution of the left $A$-module $A^t/Syz(N)$ {\small
\begin{center}
$\cdots \xrightarrow{g_{r+2}} A^{t_{r+1}}\xrightarrow{g_{r+1}} A^{t_{r}}\xrightarrow{g_r}
A^{t_{r-1}}\xrightarrow{g_{r-1}} \cdots \xrightarrow{g_2} A^{t_1}\xrightarrow{g_1}
A^{t_0}\xrightarrow{g_0} A^t/Syz(N)\longrightarrow 0$.
\end{center}}

\textit{Step 3}. We consider the complex
\begin{center}
$\cdots \xrightarrow{i\otimes g_{r+2}} A^s/Syz(M)\otimes A^{t_{r+1}}\xrightarrow{i\otimes g_{r+1}}
A^s/Syz(M)\otimes A^{t_{r}}\xrightarrow{i\otimes g_{r}} \cdots \xrightarrow{i\otimes g_2}
A^s/Syz(M)\otimes A^{t_1}\xrightarrow{i\otimes g_1} A^s/Syz(M)\otimes A^{t_0}\longrightarrow 0$,
\end{center}
where $i$ is the identical homomorphism of $A^s/Syz(M)$, and then we have the isomorphism of left
$A$-modules
\begin{center}
$Tor^{A}_r(M,N)\cong Tor^{A}_r(A^s/Syz(M),A^t/Syz(N))=\ker(i\otimes g_{r})/Im(i\otimes g_{r+1})$.
\end{center}
If $G_r$ is the matrix of $g_r$ in the canonical bases, then $\ker(i\otimes g_{r})=Syz(I_s\otimes
G_{r})$ and we get
\begin{equation}\label{tor1}
Tor^{A}_r(M,N)\cong Syz(I_s\otimes G_{r})/_A\langle I_s\otimes G_{r+1}\rangle.
\end{equation}

\textit{Step 3}. Let $q_r$ be the number of generators of $Syz(I_s\otimes G_{r})$, then by
Proposition \ref{quotient}, a presentation of $Tor^{A}_r(M,N)$ is given by
\begin{equation}\label{tor2}
Tor^{A}_r(M,N)\cong A^{q_r}/Syz(Syz(I_s\otimes G_{r})/_A\langle I_s\otimes G_{r+1}\rangle),
\end{equation}
where a set of generators of $Syz(Syz(I_s\otimes G_{r})/_A\langle I_s\otimes G_{r+1}\rangle)$ are
the first $q_r$ coordinates of generators of
\begin{equation*}\label{tor3}
Syz[Syz[I_s\otimes G_{r}]|I_s\otimes G_{r+1}].
\end{equation*}
\textit{System of generators of} $Tor^A_{r}(M,N)$, $r\geq 1$: By (\ref{tor1}), a system of
generators of $Tor^{A}_r(M,N)$ is given by a system of generators of $Syz(I_s\otimes G_{r})$. Thus,
if
\begin{center}
$Syz[I_s\otimes G_{r}]:=[\textbf{\emph{h}}_1 \cdots \textbf{\emph{h}}_{q_r}]$,
\end{center}
then
\begin{center}
$Tor^{A}_r(M,N)=_A\langle \widetilde{\textbf{\emph{h}}_1}, \dots ,
\widetilde{\textbf{\emph{h}}_{q_r}}\rangle$,
\end{center}
where $\widetilde{\textbf{\emph{h}}_v}:=\textbf{\emph{h}}_v+_A\langle I_s\otimes G_{r+1}\rangle$,
$1\leq v\leq q_r$.

\subsection{Computation of $Hom$}

\noindent In this section we will compute the right $A$-module $Hom_A(M,N)$, where \linebreak
$M:=_A\langle \textbf{\emph{f}}_1,\dots, \textbf{\emph{f}}_s\rangle$ is a finitely generated left
$A$-submodule of $A^m$ and \linebreak $N:=_A\langle
\textbf{\emph{g}}_1,\dots,\textbf{\emph{g}}_t\rangle_A$ is a finitely generated centralizing
$A$-subbimodule of $A^l$, i.e., $\textbf{\emph{g}}_i\in Cen_A(N)$, $1\leq i\leq s$. By computing
$Hom_A(M,N)$ we mean to find a presentation of $Hom_A(M,N)$ and find an specific set of generators
for $Hom_A(M,N)$. We divide the procedure in some steps.

\textit{Step 1. Presentations of $M$ and $N$.} In order to compute a presentation of $Hom_A(M,N)$
we first compute presentations of $M$ and $N$. Thus, we have
\begin{equation*}
M\cong A^s/Syz(M), N\cong A^t/Syz(N),
\end{equation*}
where $Syz(M)$ and $Syz(N)$ are the kernels of the left $A$-homomorphisms
\begin{center}
$\pi_M:A^s\longrightarrow M$ and $\pi_N:A^t\longrightarrow N$
\end{center}
defined by $\pi_M(\textbf{\emph{e}}_i):=\textbf{\emph{f}}_i$,
$\pi_N(\textbf{\emph{e}}^{\prime}_j):=\textbf{\emph{g}}_j$, $1\leq i\leq s$, $1\leq j \leq t$,
where $\{\textbf{\emph{e}}_i\}_{1\leq i\leq s}$ is the canonical basis of $A^s$ and
$\{\textbf{\emph{e}}_j'\}_{1\leq j \leq t}$ is the canonical basis of $A^t$ (see Section
\ref{sec12.5}). Since every generator $\textbf{\textit{g}}_j$ of $N$ is in $Cen_A(N)$, then $\pi_N$
is a bimodule homomorphism, $Syz(N)$ is a subbimodule of $A^t$ and hence $A^t/Syz(N)$ is an
$A$-bimodule.

Thus,
\begin{center}
$Hom_A(M,N)\cong Hom_A(A^s/Syz(M),A^t/Syz(N))$,
\end{center}
and we can compute a presentation of $Hom_A(A^s/Syz(M),A^t/Syz(N))$ instead of $Hom_A(M,N)$.

According to Section \ref{sec12.5}, $Syz(M)$ and $Syz(N)$ are computed by the syzygies of the
matrices
\begin{equation*}
F_M=
\begin{bmatrix} \textbf{\emph{f}}_1 & \cdots & \textbf{\emph{f}}_s
\end{bmatrix}, F_N=
\begin{bmatrix} \textbf{\emph{g}}_1 & \cdots & \textbf{\emph{g}}_t
\end{bmatrix},
\end{equation*}
i.e., $Syz(M)=Syz(F_M),Syz(N)=Syz(F_N)$.

\textit{Step 2. $Hom_A(A^s/Syz(M),A^t/Syz(N))$ as a kernel.} Recall that $A^s$ and $A^l$ are left
noetherian $A$-modules, so $Syz(M)$ is generated by a finite set of $s_1$ elements and $Syz(N)$ is
generated by $t_1$ elements. Therefore, we have surjective homomorphisms of left $A$-modules,
$\pi^{\prime}_M:A^{s_1}\longrightarrow Syz(M)$ and $\pi^{\prime}_N:A^{t_1}\longrightarrow Syz(N)$,
hence the following sequences of left $A$-modules are exact
\begin{equation}\label{eq1}
\begin{CD}
A^{s_1} @>{\delta_M}>> A^s @>{j_M}>> A^s/Syz(M) @>>> 0
\end{CD}
\end{equation}
\begin{equation*}\label{eq2}
\begin{CD}
A^{t_1} @>{\delta_N}>> A^t @>{j_N}>> A^t/Syz(N) @>>> 0
\end{CD}
\end{equation*}
where $\delta_M:=l_M\circ \pi^{\prime}_M$, $\delta_N:=l_N\circ \pi^{\prime}_N$, $l_M,l_N$ denote
inclusions, and $j_M,j_N$ are the canonical $A$-homomorphisms. Since $A^t/Syz(N)$ is an
$A$-bimodule, from (\ref{eq1}) we get the exact sequence of right $A$-modules
\begin{center}
$ 0\rightarrow Hom_A(A^s/Syz(M),A^t/Syz(N)) \xrightarrow{p} Hom_A(A^s,A^t/Syz(N)) \xrightarrow{d}
Hom_A(A^{s_1},A^t/Syz(N)), $
\end{center}
where
\begin{center}
$d(\alpha):=\alpha\circ \delta_M$, \,\, \text{for} \,\,$\alpha \in Hom_A(A^s,A^t/Syz(N))$;
\end{center}
$p$ is defined in a similar way. Since $Im(p)=\ker(d)$, we have the following isomorphism of right
$A$-modules:
\begin{equation}\label{eq3}
Hom_A(A^s/Syz(M),A^t/Syz(N))\cong \ker(d).
\end{equation}

\textit{Step 3. Computing finite presentations of $Hom_A(A^s,A^t/Syz(N))$ and
$Hom_A(A^{s_{1}},A^t/Syz(N))$}. By (\ref{eq3}), we must compute presentations of the right
$A$-modules $Hom_A(A^s,A^t/Syz(N))$ and $Hom_A(A^{s_1},A^t/Syz(N))$. We will show the details for
the first case, the second one is similar. We have the isomorphism of right $A$-modules (actually,
of $A$-bimodules):
\begin{center}
$Hom_A(A^s,A^t/Syz(N))\cong [A^t/Syz(N)]^s$

$f\in Hom_A(A^s,A^t/Syz(N))\mapsto (f(\textbf{\emph{e}}_1),\dots,f(\textbf{\emph{e}}_s))^T$.
\end{center}
Let $\Delta_N\in M_{t\times t_1}(A)$ be the matrix of $\delta_N$ in the canonical bases of
$A^{t_1}$ and $A^t$; we have the following isomorphism of right $A$-modules:
\begin{center}
$[A^t/Syz(N)]^s\cong A^{ts}/\langle I_s\otimes \Delta_N\rangle_A$

\bigskip

$(\overline{(a_{11},\dots,a_{t1})},\dots,(\overline{(a_{1s},\dots,a_{ts})})^T\mapsto
\overline{(a_{11},\dots,a_{t1},\dots,a_{1s},\dots,a_{ts})^T}$.
\end{center}
Hence, a presentation of $Hom_A(A^{s},A^t/Syz(N))$ is
\begin{equation}\label{315}
Hom_A(A^s,A^t/Syz(N))\cong A^{ts}/\langle I_s\otimes \Delta_N\rangle_A,
\end{equation}
where the isomorphism is defined in the following way: suppose that $f\in Hom_A(A^{s},A^t/Syz(N))$
and $f(\textbf{\emph{e}}_i):=\overline{(a_{1i},\dots, a_{ti})^T}$, $1\leq i\leq s$, then
\begin{center}
$Hom_A(A^s,A^t/Syz(N))\xrightarrow{\theta_{s,t}} A^{ts}/\langle I_s\otimes \Delta_N\rangle_A$

\bigskip

$f\mapsto \overline{(a_{11},\dots, a_{t1},\dots ,a_{1s},\dots, a_{ts})^T}$.
\end{center}
In the same way we have the isomorphism of right $A$-modules
\begin{equation}\label{315b}
Hom_A(A^{s_1},A^t/Syz(N))\xrightarrow{\theta_{s_1,t}} A^{ts_1}/\langle I_{s_1}\otimes
\Delta_N\rangle_A.
\end{equation}

\textit{Step 4. Presentation of $Hom_A(A^s/Syz(M),A^t/Syz(N))$}. From (\ref{315}) and (\ref{315b})
we define the homomorphism $\overline{d}$ of right $A$-modules by the following commutative diagram
\begin{equation*}
\begin{CD}
Hom_A(A^s,A^t/Syz(N)) @>{d}>> Hom_A(A^{s_1},A^t/Syz(N))\\
@V{\theta_{s,t}}VV @VV{\theta_{s_1,t}}V \\
A^{ts}/\langle I_s\otimes \Delta_N\rangle_A @>\overline{d}>> A^{ts_1}/\langle I_{s_1}\otimes
\Delta_N\rangle_A
\end{CD}
\end{equation*}
i.e., $\overline{d}:=\theta_{s_1,t}\circ d\circ\theta_{s,t}^{-1}$. Hence, $\ker(d)\cong
\ker(\overline{d})$, and from (\ref{eq3}), a presentation of $\ker(\overline{d})$ gives a
presentation of $Hom_A(A^s/Syz(M),A^t/Syz(N))$. We can give the explicit definition of
$\overline{d}$: if
$\overline{\textbf{\emph{a}}}:=\overline{(a_{11},\dots,a_{t1},\dots,a_{1s},\dots,a_{ts})^T}\in
A^{ts}/\langle I_s\otimes \Delta_N\rangle_A$, then{\footnotesize
\begin{center}
$\overline{d}(\overline{\textbf{\emph{a}}})=\overline{(\sum_{k=1}^{s}\delta_{k1}a_{1k},\dots,\sum_{k=1}^{s}\delta_{k1}a_{tk},
\dots,
\sum_{k=1}^{s}\delta_{ks_1}a_{1k},\dots,\sum_{k=1}^{s}\delta_{ks_1}a_{tk})}=\overline{(\Delta_M^T\otimes
I_t)\textbf{\emph{a}}}$,
\end{center}}
\noindent where $\Delta_M:=[\delta_{ij}]\in M_{s\times s_1}(A)$ is the matrix of $\delta_M$ in the
canonical bases. In order to apply the right version of results of the previous chapter, we have to
interpret $A$ as a right bijective skew $PBW$ extension of the $GS$ ring $R$ and rewrite the
entries of matrices over $A$ with the coefficients in the right side. With this, note that{\tiny
\begin{equation*}
\overline{\textbf{\emph{a}}}\in \ker(\overline{d}) \Longleftrightarrow
\overline{d}(\overline{\textbf{\emph{a}}}) = \overline{\textbf{0}} \Longleftrightarrow
(\Delta_M^T\otimes I_t)\textbf{\emph{a}}\in \langle I_{s_1}\otimes
\Delta_N\rangle_A\Longleftrightarrow
\end{equation*}
\begin{equation*}
\text{\textit{the coordinates of $\textbf{a}$ are the first} $st$ \textit{entries of some element
of} $\textrm{Syz}^r\,([(\Delta_M\otimes I_t)^{T}|I_{s_1}\otimes \Delta_N)])$}
\end{equation*}
\begin{equation*}
\Longleftrightarrow \overline{\textbf{\emph{a}}}\in \langle U\rangle_A/\langle I_s\otimes
\Delta_N\rangle_A,
\end{equation*}}
where $\langle U\rangle_A$ is the right $A$-module generated by the columns of the matrix $U$
defined by
\begin{align*}
\text{\textit{columns of} $U$}:=&\text{\textit{first} $st$}\,\,  \text{\textit{coordinates of generators of}}\\
&Syz^r\,([(\Delta_M\otimes I_t)^{T}|I_{s_1}\otimes \Delta_N)]).
\end{align*}
Thus, we have proved that $\ker(\overline{d})=\langle U\rangle_A/\langle I_s\otimes
\Delta_N\rangle_A$, and we get the following theorem.

\begin{theorem}\label{homtheorem}
With the notation above,
\begin{equation}\label{hom}
Hom_A(M,N)\cong \langle U\rangle_A/\langle I_s\otimes \Delta_N\rangle_A,
\end{equation}
and a presentation of $\langle U\rangle_A/\langle I_s\otimes \Delta_N\rangle_A$ is a presentation
for $Hom_A(M,N)$.
\end{theorem}
Observe that a system of generators of $\langle U\rangle_A$ gives a system of generators for
$Hom_A(M,N)$.

\subsection{Computation of $Ext$}\label{10.7} Using syzygies we now describe an easy procedure for
computing the right $A$-modules $Ext^r_{A}(M,N)$, for $r\geq 0$, where $M:=_A\langle
\textbf{\emph{f}}_1,\dots, \textbf{\emph{f}}_s\rangle$ is a finitely generated left $A$-submodule
of $A^m$ and $N:=_A\langle \textbf{\emph{g}}_1,\dots,\textbf{\emph{g}}_t\rangle_A$ is a finitely
generated centralizing $A$-subbimodule of $A^l$, i.e., $\textbf{\emph{g}}_i\in Cen_A(N)$, $1\leq
i\leq s$. By computing $Ext^r_{A}(M,N)$ we mean to find a presentation of $Ext^r_{A}(M,N)$ and a
system of generators.

For $r=0$, $Ext^0_{A}(M,N)=Hom_A(M,N)$ and the computation was given in the previous section. So we
assume that $r\geq 1$.

\textit{Presentation of} $Ext^r_{A}(M,N)$, $r\geq 1$:

\textit{Step 1}. As in the previous section, we compute presentations of $M$ and $N$,
\begin{center}
$M\cong A^s/Syz(M),N\cong A^t/Syz(N)$.
\end{center}

\textit{Step 2}. We compute a free resolution of the left $A$-module $A^s/Syz(M)$,{\footnotesize
\begin{center}
$\cdots \xrightarrow{f_{r+2}} A^{s_{r+1}}\xrightarrow{f_{r+1}} A^{s_{r}}\xrightarrow{f_r}
A^{s_{r-1}}\xrightarrow{f_{r-1}} \cdots \xrightarrow{f_2} A^{s_1}\xrightarrow{f_1}
A^{s_0}\xrightarrow{f_0} A^s/Syz(M)\longrightarrow 0$.
\end{center}}
For every $r\geq 1$, let $F_r$ be the matrix of $f_r$ in the canonical bases.

\textit{Step 3}. As we observed in the previous section, $A^t/Syz(N)$ is an $A$-bimodule, so we get
the complex of right $A$-modules
\begin{center}
$0 \longrightarrow Hom_A(A^{s_0},A^t/Syz(N))\xrightarrow{f_{1}^*} \cdots \xrightarrow{f_{r}^*}
Hom_A(A^{s_r},A^t/Syz(N))\xrightarrow{f_{r+1}^*}
Hom_A(A^{s_{r+1}},A^t/Syz(N))\xrightarrow{f_{r+2}^{*}} \cdots$
\end{center}
According to Remark \ref{12.7.7}, we can interpret $A$ as a right bijective skew $PBW$ extension of
$R$, and hence, we can apply any of results of the previous section but in the right version.
Recall that{\small
\begin{equation}\label{ext1}
Ext^r_{A}(M,N)\cong Ext^r_{A}(A^s/Syz(M),A^t/Syz(N))=\ker(f_{r+1}^*)/Im(f_{r}^*).
\end{equation}}
By (\ref{315}), for each $r\geq 1$, a presentation of $Hom_A(A^{s_r},A^t/Syz(N))$ is given by the
following isomorphism of right $A$-modules
\begin{center}
$Hom_A(A^{s_r},A^t/Syz(N))\cong A^{ts_r}/\langle I_{s_r}\otimes Syz(N)\rangle_A$.
\end{center}
\textit{Step 4}. We can associate a matrix $F_r^*$ to the homomorphism $f_{r}^*$, as follows: we
have the following commutative diagram
\begin{equation*}
\begin{CD}
Hom_A(A^{s_r},A^t/Syz(N)) @>{f_{r+1}^*}>> Hom_A(A^{s_{r+1}},A^t/Syz(N)) \\
@VVV @VVV \\
A^{ts_r}/K_r @>{\overline{f_{r+1}^*}}>> A^{ts_{r+1}}/K_{r+1}
\end{CD}
\end{equation*}
where $K_r:=\langle I_{s_r}\otimes Syz(N)\rangle_A$ and the vertical arrows are isomorphisms of
right $A$-modules obtained concatenating the columns of matrices of \linebreak
$Hom_A(A^{s_r},A^t/Syz(N))$; moreover, the matrices of homomorphisms $f_{r+1}^*$ and
$\overline{f_{r+1}^*}$ coincides. So, we can replace the above complex for the following equivalent
complex
\begin{center}
$0 \longrightarrow A^{ts_0}/K_0\xrightarrow{f_1^*} \cdots \xrightarrow{f_{r}^*}
A^{ts_r}/K_r\xrightarrow{f_{r+1}^*} A^{ts_{r+1}}/K_{r+1}\xrightarrow{f_{r+2}^{*}} \cdots$.
\end{center}
We will compute the matrix $F_{r+1}^*$: let $\{\textbf{\emph{e}}_1,\dots
,\textbf{\emph{e}}_{ts_r}\}$ be the canonical basis of $A^{ts_r}$, then for each $1\leq i\leq
ts_r$, the element $\overline{\textbf{\emph{e}}_i}=\textbf{\emph{e}}_i+K_r$ can be replaced by its
corresponding canonical matrix $G_i$, and since $f_{r+1}^*(g)=gf_{r+1}$ for each $g\in
Hom_A(A^{s_r},A^t/Syz(N))$, then we conclude that
\begin{center}
$F_{r+1}^*=I_t\otimes F_{r+1}^T$.
\end{center}

\textit{Step 5}. By (\ref{ext1}), a presentation of $Ext^r_{A}(M,N)$ is given by a presentation of
$\ker(f_{r+1}^*)/Im(f_{r}^*)$. Hence, we can apply Proposition \ref{quotient}, let $p_r$ be the
number of generators of $\ker(f_{r+1}^*)=Syz^r(F_{r+1}^*)=Syz^r(I_t\otimes F_{r+1}^T)$, where the
entries of the matrix $F_{r+1}$ should be rewritten with coefficients in the right side; we know
how to compute this syzygy, and hence, we know how to compute $p_r$. We also know how to compute
the matrix $F_{r}^*=I_t\otimes F_{r}^T$. Then, a presentation of $Ext^r_{A}(M,N)$ is given by
\begin{equation}\label{ext2}
Ext^r_{A}(M,N)\cong A^{p_r}/Syz(\ker(f_{r+1}^*)/Im(f_{r}^*)),
\end{equation}
where a set of generators of $Syz(\ker(f_{r+1}^*)/Im(f_{r}^*))$ are the first $p_r$ coordinates of
the generators of
\begin{equation}\label{ext3}
Syz^r[Syz^r[I_t\otimes F_{r+1}^T]|I_t\otimes F_{r}^T].
\end{equation}

\textit{System of generators of} $Ext^r_{A}(M,N)$, $r\geq 1$: By (\ref{ext1}), a system of
generators for $Ext^r_{A}(M,N)$ is defined by a system of generators of
$\ker(f_{r+1}^*)=Syz^r(F_{r+1}^*)=Syz^r(I_t\otimes F_{r+1}^T)$, hence if
\begin{center}
$Syz^r[I_t\otimes F_{r+1}^T]:=[\textbf{\emph{h}}_1 \cdots \textbf{\emph{h}}_{p_r}]$,
\end{center}
then
\begin{center}
$Ext^r_{A}(M,N)=\langle \widetilde{\textbf{\emph{h}}_1},\dots
,\widetilde{\textbf{\emph{h}}_{p_r}}\rangle_A$,
\end{center}
where $\widetilde{\textbf{\emph{h}}_u}:=\textbf{\emph{h}}_u+Im(F_r^*)$, $1\leq u\leq p_r$.

\begin{remark}
(i) Observe that if we take $M=A$, then $M\otimes N\cong N$ and this agrees with the conclusion of
Theorem \ref{tensor}: in fact, in this trivial case, $s=1$ and we have
\begin{equation*}
Syz(M\otimes N)\cong Syz(N);
\end{equation*}{\small
\begin{equation*}
Syz(M\otimes N)=_A\langle [Syz(M)\otimes I_t\mid I_s\otimes Syz(N)]\rangle=_A\langle [0\otimes
I_t\mid I_1\otimes Syz(N)]\rangle\cong Syz(N).
\end{equation*}}
Thus, a presentation of $M\otimes N$ is
\begin{equation*}
N\cong M\otimes N\cong A^{st}/Syz(M\otimes N)=A^{t}/Syz(A\otimes N)\cong A^t/Syz(N).
\end{equation*}
(ii) For $M=A$ and $r\geq 1$, $Tor_r^A(M,N)=0$ and this agrees with (\ref{tor1}) and (\ref{tor2})
since in this case $s=1$.

(iii) Taking $M=A$ in $Hom_A(M,N)$ we have $Hom_A(M,N)\cong N$ (isomorphism of right $A$-modules);
on the other hand, from Theorem \ref{homtheorem} we get the isomorphism of right $A$-modules
\begin{equation*}
Hom_A(M,N)\cong \langle U\rangle_A/\langle I_1\otimes \Delta_N\rangle_A=\langle U\rangle_A/\langle
\Delta_N\rangle_A,
\end{equation*}
where the columns of the matrix $U$ are defined by
\begin{center}
$\text{\textit{columns of\, U} }=\text{\textit{first\, t} }\,\,  \text{\textit{coordinates of
generators of}}$

$Syz^r\,([(\Delta_M\otimes I_t)^{T}|I_{s_1}\otimes \Delta_N)])=Syz^r\,([(0\otimes
I_t)^{T}|I_{0}\otimes \Delta_N)])= Syz^r\,(\Delta_N)=Syz^r\,([(0\otimes I_t)^{T}|0\otimes
\Delta_N)])= Syz^r\,(0)=A^t$.
\end{center}
Thus, $N\cong Hom_A(M,N)\cong A^t/\langle\Delta_N\rangle_A$ is the finite presentation of $N$ as
right $A$-module.

(iv) If $M=A$, then $Ext_A^r(M,N)=0$ for $r\geq 1$ and this agrees with (\ref{ext2}) and
(\ref{ext3}) since in this case $f_i=0$ and $F_i=0$ for $i\geq 1$.
\end{remark}

\section{Some applications}

\noindent As application of the computations of $Hom$ and $Ext$ we will test stably-freeness,
reflexiveness, and we will compute also the torsion, the dual and the grade of a given submodule of
$A^m$. For these applications the centralizing bimodule we need is trivial, $A=_A\langle
1\rangle_A$. We will illustrate these applications with examples, and for this we will use some
computations we found in \cite{gallego-thesis}.

\begin{example}\label{exmp1}
The first application is a method for testing stably-freeness. It is well known that if $S$ is a
ring and $M$ is a $S$-module with exact sequence $0\rightarrow
S^s\xrightarrow{f_1}S^r\xrightarrow{f_0}M\rightarrow 0$, $M$ is stably-free if and only if
$Ext_S^{1}(M,S)=0$ (see \cite{Chyzak3}). We can illustrate this method with the following example:
Consider the skew $PBW$ extension $A:=\sigma(\mathbb{Q})\langle x,y\rangle$, with $yx=-xy$. We will
check if the following left $A$-module is stably-free
\begin{center}
$M:=\langle
\textbf{\emph{e}}_3+\textbf{\emph{e}}_1,\textbf{\emph{e}}_4+\textbf{\emph{e}}_2,x\textbf{\emph{e}}_2+
x\textbf{\emph{e}}_1,y\textbf{\emph{e}}_1,y^2\textbf{\emph{e}}_4,x\textbf{\emph{e}}_4+y\textbf{\emph{e}}_3
\rangle$.
\end{center}
On $Mon(A)$ we consider the deglex order with $x\succ y$ and the $TOP$ order on $Mon(A^4)$ with
$\textbf{\emph{e}}_4>\textbf{\emph{e}}_3>\textbf{\emph{e}}_2>\textbf{\emph{e}}_1$; according to
\cite{gallego-thesis}, a system of generators for $Syz(M)$ is
\[S=\{(0,-xy^2,y^2,-xy,x,0), (-y^2,xy,y,x+y,0,y), (y^3,0,0,-y^2,x,-y^2)\}.\]
From this we get that a finite presentation of $M$ is given by
\begin{equation*}
\begin{CD}
A^3 @>{F_1}>> A^6@>{F_0}>>M\longrightarrow 0,
\end{CD}
\end{equation*}
where
\begin{center}
$F_1:=\begin{bmatrix}0 &-y^2&y^3 \\ -xy^2&xy&0\\
y^2&y&0\\-xy&x+y&-y^2\\x&0&x\\0&y&-y^2\end{bmatrix}$ and  $F_0:=\begin{bmatrix}
1&0&x&y&0&0\\0&1&x&0&0&0\\1&0&0&0&0&y\\0&1&0&0&y^2&x \end{bmatrix}$.
\end{center}
Applying again the method for computing syzygies we get that $Syz(F_1)=0$ and hence we obtain
\begin{equation*}
\begin{CD}
0\longrightarrow A^3 @>{F_1}>> A^6@>{F_0}>>M\longrightarrow 0.
\end{CD}
\end{equation*}
From (\ref{ext2}), a presentation for $Ext^1_{A}(M,A)$ is given by
\begin{equation*}
Ext^1_{A}(M,A)\cong A^{p_r}/Syz(ker(f_{2}^{*})/Im(f_{1}^{*})),
\end{equation*}
and from (\ref{ext3}) a system of generators for $Syz(ker(f_{2}^{*})/Im(f_{1}^{*}))$ are the first
$p_r$ coordinates of the generators of
\begin{equation*}
Syz^r[Syz^r[I_t\otimes F_{2}^T]|I_t\otimes F_{1}^T],
\end{equation*}
and $p_r$ is the number of generators of $Syz^r[I_t\otimes F_{2}^T]$. In our particular situation
$t=1$ and
\begin{equation*}
F_2=\begin{bmatrix} 0 \\ 0 \\ 0 \end{bmatrix},
\end{equation*}
whence $Syz^r[I_t\otimes F_{2}^T]=I_3$ and $p_r=3$. Therefore,
\begin{eqnarray*}
Syz^r[Syz^r[I_t\otimes F_{2}^T]|I_t\otimes F_{1}^T]=Syz^r\begin{bmatrix}1 & 0 & 0 &  0   & -xy^2 & y^2 & -xy & x & 0 \\
0 & 1 & 0 & -y^2 & xy    & y   & x+y & 0 & y \\
0 & 0 & 1 & y^3  & 0 & 0 & -y^2 & x & -y^2 \end{bmatrix}.
\end{eqnarray*}
In order to compute this right syzygy, we first have to compute a Gröbner basis $G$ for
$$F:=\begin{bmatrix}1 & 0 & 0 &  0   & -xy^2 & y^2 & -xy & x & 0 \\
                                                                           0 & 1 & 0 & -y^2 & xy    & y   & x+y & 0 & y \\
                                                                           0 & 0 & 1 & y^3  & 0 & 0 & -y^2 & x & -y^2 \end{bmatrix},$$
but this task is trivial since the first three vectors are the canonical vectors of $A^3$, so
$G=\{\textbf{\emph{e}}_1,\textbf{\emph{e}}_2,\textbf{\emph{e}}_3\}$ and hence $Syz^r(G)=0$.
According to the procedure for computing $Syz^r(F)$ (the right version of the procedure for
computing syzygies, see \cite{gallego-thesis} and \cite{Jimenez2}), we have $G=FH$, $F=GQ$, with
\begin{equation*}
H=\begin{bmatrix}
1 &0 & 0 \\
0 &1 & 0 \\
0 &0 & 1 \\
0 &0 & 0 \\
0 &0 & 0 \\
0 &0 & 0 \\
0 &0 & 0 \\
0 &0 & 0 \\
0 &0 & 0 \\
\end{bmatrix}
,\, Q=F,
\end{equation*}
and
\begin{center}
$Syz^r(F)=\begin{bmatrix}HSyz^r(G)|I_s-HQ\end{bmatrix}=I_s-HQ$.
\end{center}
Thus,
\begin{equation*}
I_9 - HQ=\begin{bmatrix}
0    &0    &0     &0    &  xy^2 & -y^2 & xy & -x & 0   \\
0    &0    &0     &y^2  &  -xy  & -y    & -x-y & 0 & -y  \\
0    &0    &0     &-y^3 &  0 & 0  & y^2 & -x& y^2 \\
0    &0  &0 &1    &  0 & 0     & 0 & 0& 0    \\
0 &0  &0     &0    &  1 & 0     & 0 & 0& 0    \\
0  &0   &0     &0    &  0 & 1     & 0 & 0& 0   \\
0   &0 &0   &0    &  0 & 0     & 1 & 0& 0  \\
0   &0    &0   &0    &  0 & 0     & 0 & 1& 0   \\
0    &0   &0   &0    &  0 & 0     & 0 & 0& 1     \\
\end{bmatrix},
\end{equation*}
and
\begin{equation*}
Ext^1_{A}(M,A)\cong A^3/\langle S\rangle_A,
\end{equation*}
where{\footnotesize
\begin{center}
$S:=\{(0,y^2,-y^3),(xy^2,-xy,0),(-y^2,-y,0),(xy,-x-y,y^2),(-x,0-x),(0,-y,y^2)\}$.
\end{center}}
This proves that $Ext^1_{A}(M,A)\neq 0$, and hence, $M$ is not stably-free.
\end{example}

\begin{example}\label{exmp2}
(i) The second application consists in using Theorem \ref{homtheorem} in order to compute a
presentation for the dual right $A$-module $M^{*}=Hom_A(M,A)$, where $M$ is a left $A$-submodule of
$A^m$. Let $M \subseteq A^4$ be as in the previous example. By (\ref{hom}),
\begin{equation*}
Hom_A(M,A)\cong \langle U\rangle_A/\langle I_6\otimes
\Delta_A\rangle_A,
\end{equation*}
where $\langle U\rangle_A$ is the right $A$-module generated by the columns of the matrix $U$, and
the columns of $U$ are the first $6$ coordinates of the generators of
$$Syz^r\,([(\Delta_M\otimes I_1)^{T}|I_{3}\otimes \Delta_A)]),$$
with $\Delta_M=F_1$ and $\Delta_A=[1]$. Thus, we have to compute $Syz^r\,([F_1^T|I_{3}])$. In a
similar manner as in the the previous example, $Syz^r\,([F_1^T|I_{3}])$ is generated by the columns
of $I_9 - HQ$, where
\begin{center}
$Q=F:=\begin{bmatrix}0 & -xy^2 & y^2 & -xy & x & 0 & 1 & 0 & 0\\
-y^2 & xy & y & x+y & 0 & y & 0 & 1 & 0\\
y^3 & 0 & 0 & -y^2 & x & -y^2 & 0 & 0 & 1\end{bmatrix}$
\end{center}
and
\begin{equation*}
H=\begin{bmatrix}
0 & 0 & 0\\
0 & 0 & 0\\
0 & 0 & 0\\
0 & 0 & 0\\
0 & 0 & 0\\
0 & 0 & 0\\
1 & 0 & 0\\
0 & 1 & 0\\
0 & 0 & 1
\end{bmatrix}
\end{equation*}
Therefore,
\begin{equation*}
I_9 -HQ=\begin{bmatrix}
1 & 0 & 0 & 0 & 0 & 0 & 0 & 0  & 0 \\
0 & 1 & 0 & 0 & 0 & 0 & 0 & 0  & 0 \\
0 & 0 & 1 & 0 & 0 & 0 & 0 & 0  & 0 \\
0 & 0 & 0 & 1 & 0 & 0 & 0 & 0  & 0 \\
0 & 0 & 0 & 0 & 1 & 0 & 0 & 0  & 0 \\
0 & 0 & 0 & 0 & 0 & 1 & 0 & 0  & 0 \\
0 & xy^2 & -y^2 & xy & -x & 0 & 0 & 0 & 0 \\
y^2 & -xy & -y & -x-y & 0 & -y & 0 & 0 & 0 \\
-y^3 & 0 & 0 & y^2 & -x & y^2 & 0 & 0 & 0
\end{bmatrix}.
\end{equation*}
Hence, $U=I_6$ and $M^*=Hom_A(M,A)=0$.

(ii) Our next application is about the grade of a module. Let $S$ be a ring, the \textit{grade}
$j(M)$ of a left $S$-module $M$ is defined by
\begin{center}
$j(M):=\min\{i|Ext_S^i(M,S)\neq 0\}$,
\end{center}
or $\infty$ if no such $i$ exists. Therefore, we can compute the grade of a given submodule
$M\subseteq A^m$. Thus, according to (i) and the previous example, $Ext^0_A(M,A)=0$ but
$Ext^1_A(M,A) \neq 0$, where $A=\sigma(\mathbb{Q})\langle x,y\rangle$, with $yx=-xy$ and $M=\langle
\textbf{\emph{e}}_3+\textbf{\emph{e}}_1,\textbf{\emph{e}}_4+\textbf{\emph{e}}_2,x\textbf{\emph{e}}_2+
x\textbf{\emph{e}}_1,y\textbf{\emph{e}}_1,y^2\textbf{\emph{e}}_4,x\textbf{\emph{e}}_4+y\textbf{\emph{e}}_3
\rangle \subseteq A^4$. Whence, $j(M)=1$.

(iii) Now suppose that the ring $R$ of coefficients of the extension $A$ is a noetherian domain
(left and right). Then $A$ is also a noetherian domain (see \cite{Lezama3}). Let $M$ be a left
$A$-module given by the presentation $A^s\xrightarrow{F_1} A^r\xrightarrow{F_0} M\rightarrow 0$.
The $Ext$ modules and some results of \cite{Chyzak3} can be used for computing the torsion
submodule $t(M)$, and also they can be applied for testing reflexiveness of $M$ in the particular
case when $M\subseteq A^m$. In fact, Theorems 5 of \cite{Chyzak3} states that $t(M)\cong
Ext_A^1(M^T,A)$, where $M^T$ is the transposed module of $M$, $M^T=A^s/Im(F_1^{T})$, i.e., $M^T$ is
given by the presentation $A^r\xrightarrow{F_1^T} A^s\rightarrow M^T\rightarrow 0$. From this we
get that $M$ is torsion-free if and only if $ Ext_A^1(M^T,A)=0$, and also, $M$ is a torsion module
if and only if $ Ext_D^1(M^T,A)=M$.

Let $M$ and $A$ be again as in Example \ref{exmp1}. For $M^T$ we have the finite presentation
\begin{equation*}
\begin{CD}
0\longrightarrow A^6 @>{F_1^T}>> A^3@>{}>>M^T\longrightarrow 0;
\end{CD}
\end{equation*}
moreover,
\begin{equation*}
Syz^r[I_1\otimes F_{2}^T]=Syz^r[0]=I_6,
\end{equation*}
so
\begin{equation*}
Syz^r[Syz^r[I_t\otimes F_{2}^T]|I_1\otimes (F_{1}^T)^T]= Syz^r[I_6 \vert F_1].
\end{equation*}
Reasoning as in Example \ref{exmp1} (see also (i)), in this case we have $Q=F:=[I_6\vert F_1]$ and
$H=\begin{bmatrix}I_6\\ 0\end{bmatrix}$, where $0$ is a null matrix of size $3\times 6$, so we
compute
\begin{equation*}
I_9 - HQ=\begin{bmatrix}
0    &0    &0     &0    &  0 & 0     & 0 & y^2& -y^3   \\
0    &0    &0     &0    &  0 & 0     & xy^2 & -xy & 0  \\
0    &0    &0     &0    &  0 & 0     & -y^2 & -y & 0  \\
0    &0    &0     &0    &  0 & 0     & xy & -x-y& y^2    \\
0    &0    &0     &0    &  0 & 0     & -x & 0& -x    \\
0    &0    &0     &0    &  0 & 0     & 0 & -y & y^2   \\
0    &0    &0     &0    &  0 & 0     & 1 & 0 & 0  \\
0    &0    &0     &0    &  0 & 0     & 0 & 1 & 0  \\
0    &0    &0     &0    &  0 & 0     & 0 & 0 & 1
\end{bmatrix}.
\end{equation*}
$Syz^r(F)$ is generated by the columns of $I_9 - HQ$ and $Ext^1_A(M^T,A) \cong A^6/J$, where $J$ is
the right $A$-module generated by the first six coordinates of the generators of $Syz^r(F)$, so
$J=Syz^r(M)$, and hence
\begin{equation*}
Ext^1_A(M^T,A) \cong A^6/Syz^r(M) \cong M.
\end{equation*}
(iv) Observe that another interesting method for checking whether $M$ is a torsion module is
Corollary 1 of \cite{Chyzak3}: $M$ is a torsion module if and only if $M^*=0$. Let $M$ and $A$ be
as in the previous examples; in (i) we showed that $M^*=0$, so $M$ is a torsion module and this
agrees with (iii).

(v) Finally, recall that $M$ is \textit{reflexive} if $(M^*)^*\cong M$. Theorem 6 of \cite{Chyzak3}
proves that $M$ is reflexive if and only if $Ext_A^i(M^T,A)=0$, for $i=1,2$. According to (iii),
$M$ is not reflexive.
\end{example}



\begin{flushright}
Departamento de Matemáticas\\
Universidad Nacional de Colombia\\
Bogotá, Colombia\\
\textit{e-mail}: \texttt{jolezamas@unal.edu.co}\\
\end{flushright}
\end{document}